\documentclass[12pt,nopagetitle]{article}
\usepackage{amsfonts, amsmath, amssymb, amsthm, mathtools}  
\usepackage[english]{babel}
\usepackage{comment}
\usepackage{color}
\usepackage{graphicx, wrapfig} 
\usepackage{hyperref}

\newcounter{iii}                                            

\theoremstyle{plain}
\newtheorem{theorem}{Theorem}

\newtheorem{lemma}{Lemma}
\newtheorem{fact}{Fact}
\newtheorem{problem}{Problem}

\newtheorem{conjecture}{Conjecture}

\theoremstyle{definition}

\author{Alexandr Polyanskii}
\title{Helly-type theorem for eigenvectors}
\date{\today}

\begin{document}
\maketitle
\abstract{We prove that if any $\lfloor3d/2 \rfloor$ or fewer elements of a finite family of linear operators $\mathbb K^d\to \mathbb K^d$ ($\mathbb K$ is an arbitrary field) have a common eigenvector then all operators in the family have a common eigenvector. Moreover, $\lfloor 3d/2\rfloor$ cannot be replaced by a smaller number. Also, we study the following problem, achieving partial results: prove that if any $l=O(d)$ or fewer elements of a finite family of linear operators $\mathbb K^d\to \mathbb K^d$ have a common non-trivial invariant subspace then all operators in the family have a common non-trivial invariant subspace. 
}
	
 \section{Introduction}
	We denote the set of numbers $\{a,a+1,\dots, b\}$ by $[a,b]$, where $a<b$ are positive integer numbers. Also, we use the following notation $[b]:=[1,b]$.
    
The well-known Helly's theorem~\cite{Helly} claims.
\begin{theorem}
\label{theorem:helly}
Let $\mathcal F$ be a finite family of convex sets in $\mathbb R^d$. If every $d + 1$ or fewer elements of $\mathcal F$ intersects, then all the sets in the family $\mathcal F$ intersect.
\end{theorem}

There is an abundance of literature on Helly-type theorems, see for example \cite{ALS, DGK63,Eck,Mat,W}. The aim of this work is to answer the following Helly-type question.
\begin{problem}
\label{problem:hellyeigenvectors}
Let $\mathcal{A}$ be a finite family of linear operators $\mathbb K^d\to \mathbb K^d$, where $\mathbb K$ is an arbitrary field and $d\geqslant 2$. Find minimal $k=HE(\mathbb K^d)$ such that if any $k$ or fewer elements of $\mathcal A$ have a common eigenvector then there is a common eigenvector for all operators in the family~$\mathcal{A}$. 
	\end{problem}	
The authors of \cite{B} considered Problem~\ref{problem:hellyeigenvectors} and gave the wrong proof of the fact that $HE(\mathbb K^d)=d+1$ (see Theorem~5.1 in~\cite{B}). Our main result is the following theorem.
\begin{theorem}
\label{theorem:hellyeigenvectors}
$HE(\mathbb K^d)=\lfloor3d/2\rfloor$ for an arbitrary field $\mathbb K$.
\end{theorem}
We prove Theorem~\ref{theorem:hellyeigenvectors} in Section~\ref{section:proofoftheoremhellyeigenvectors}.

Also, we discuss the following Helly-type question posed by Andrey Voynov~\cite{AV}.

\begin{problem}
\label{problem:invariantsubspaces}
Let $\mathcal{A}$ be a finite family of linear operators $\mathbb K^d\to \mathbb K^d$, where $\mathbb K$ is an arbitrary field and $d\geqslant 2$. Find minimal $l=HI(\mathbb K^d)$ such that if any $l$ or fewer elements of $\mathcal{A}$ have a common non-trivial invariant subspace then there is a common non-trivial invariant subspace for all operators in the family~$\mathcal{A}$. 
\end{problem}	
It is easy to see that $HI(\mathbb K^d)\leq d^2$. Indeed, suppose that we know that any $d^2$ or fewer operators in a finite family have a common non-trivial invariant subspace. Note that the space of linear operators is in fact the space of matrices of size $d\times d$ therefore among operators in the family we can find at most $d^2$ operators such that any operator in the family is equal to a linear combination of these operators (with coefficients in $\mathbb K$. Therefore, because of our assumption all elements of the family have a common non-trivial subspace.

It is not difficult to slightly improve the above estimate $HI(\mathbb K^d)\leq d^2-d+1$. We leave this improvement to reader as an exercise. Moreover, we believe that the following conjecture holds.

\begin{conjecture}\label{conjecture:invariantsubspacesarelinear}
$HI(\mathbb K^d)=O(d)$.
\end{conjecture}

The following theorem is a partial result confirming Conjecture~\ref{conjecture:invariantsubspacesarelinear}.

\begin{theorem}\label{theorem:hellyinvariansubspaces}
Suppose that $\mathcal{A}$ is a finite family of linear operators such that $A_0\in \mathcal{A}$ have $d$ eigenvectors in $\mathbb K^d$ with different eigenvalues and any $2d-1$ or fewer elements of $\mathcal{A}$ have a common non-trivial invariant subspace then there is a common non-trivial invariant subspace for all elements of $\mathcal{A}$. 
\end{theorem}
We deduce Theorem~\ref{theorem:hellyinvariansubspaces} from the following combinatorial lemma, which seems to be of independent interest.
\begin{lemma}
\label{lemma:subsets}
	Let $q>0$ and $p\geq 0$ be integer numbers. Given a family $\mathcal{M}=\{M_1, \dots, M_p\}$ of non-empty and proper subsets of $[q]$. Suppose that for every non-empty subset $I\subseteq [p]$ we have either
    $$\bigcup_{i\in I}M_i=[q] \text{ or } \bigcup_{i\in I} M_i\ne \bigcup_{i\in I\setminus\{j\}}M_i$$ for some $j\in I$. Then $p\leqslant 2q-2$.   
\end{lemma}
Theorem~\ref{theorem:hellyinvariansubspaces} and Lemma~\ref{lemma:subsets} are proved in Section~\ref{section:prooftheoremhellyinvariantsubspaces}. 

{\bf Acknowledgment}. We are grateful to Roman Drnov\u{s}ek for valuable comments that helped to significantly improve the presentation of the paper. We wish to thank Ilya Bogdanov for pointing out a mistake in the original proof of~Lemma~\ref{lemma:subsets}.

\section{Proof of Theorem \ref{theorem:hellyeigenvectors}} 
\label{section:proofoftheoremhellyeigenvectors}
	
The following example shows that $HE(\mathbb K^d)\geq \lfloor 3d/2\rfloor$.
	
	\emph{Example.} Let $\mathbf e_1, \dots, \mathbf e_{3n}$ be such vectors that 
	$$\{\mathbf e_j: j\in[3n], 3\nmid j\}$$ 
	is the standard basis vectors of $\mathbb K^{2n}$ and $\mathbf e_{3i}=\mathbf e_{3i-2}+\mathbf e_{3i-1}$ for any $i\in[n]$. Let us introduce the following notations: $$H_j=span (\mathbf e_{i}:i\in[3n]\setminus \{j, j+f(j)\}), L_j=span(\mathbf e_j),$$ 
here $j\in[3n]$ and
$$f(j)=\begin{cases}1, \text{ if } j\equiv1,2&\pmod{3};
\\-2,\text{ if } j\equiv 0&\pmod{3}.
\end{cases}$$
Let us define a family $\mathcal A=\{A_1,\dots, A_{3n}\}$ of $3n$ operators. The operator $A_{j}$, $j\in [3n]$, is such that $H_j$ is the eigenspace associated with the eigenvalue $1$ and $L_j$ is the eigenspace associated with the eigenvalue $0$.
	Obviously, $A_j$ are well defined and all operators but $A_j$ have a common eigenvector $\mathbf e_{j+f(j)} $, but all operators do not have a common eigenvector. Analogously, if $d$ is odd we can construct an example of a family of $\lfloor3d/2\rfloor$ operators $\mathbb K^d\to\mathbb K^d$ such that all operators does not have a common eigenvector and all operators but any operator have.
\\[2pt]

The proof is based on induction on $n$. Suppose that Theorem~\ref{theorem:hellyeigenvectors} is shown for families of linear operators containing less then $n$ operators, where $n\geqslant \lfloor3d/2\rfloor +1$. Let us prove Theorem~\ref{theorem:hellyeigenvectors} for a family $\mathcal{A}=\{A_1, \dots, A_n\}$ of $n$ operators. By the induction hypothesis, we can find a vector $\mathbf v_i$ that is a common eigenvector of all elements of $\mathcal{A}$ but $A_i$.
		
Suppose that $\mathbf v_{k+1}, \dots, \mathbf v_{n}$ is the maximal set of linearly independent vectors among $\mathbf v_1, \dots, \mathbf v_n$, then $n-k\leqslant d$ and
\begin{equation}\label{e3}
\mathbf v_i=\sum_{j=k+1}^n \mu_{i,j}\mathbf v_j=\sum_{j\in X_i} \mu_{i,j} \mathbf v_j
\end{equation}
for any $i\in[k]$. Here $X_i$ is the set of $j\in[k+1,n]$ such that $\mu_{i,j}\ne0$.
 \begin{fact} \label{f2}
$H_i:=span (\{\mathbf v_j:j\in X_i\})$ is a common eigenspace for operators $A_m$, $m\in [n]\setminus (X_i\cup \{i\})$.
\end{fact}
\begin{proof}[Proof of Fact \ref{f2}] Let us apply the operator $A_m$ for fixed $m\in [n]\setminus (X_i\cup \{i\})$ to (\ref{e3}). Because vectors $\mathbf v_j$, $j\in X_i\cup \{i\}$, are eigenvectors of $A_m$ thus we get
$$\lambda_{i,m} \mathbf v_i=\sum_{j\in X_i}\lambda_{j,m}\mu_{i,j}\mathbf v_j,$$
where $\lambda_{j,m}$, $j\in\{i\}\cup X_i$, is the eigenvalue corresponding to the operator $A_m$ and the eigenvector $\mathbf v_j$. Because vectors $\mathbf v_j$, $j\in X_j\subset [k+1, n]$, are linearly independent therefore $\lambda_{j,m}=\lambda$ for all $j\in\{i\}\cup X_i$. Thus the operator $A_m$ is a scalar operator on $H_i$.
\end{proof}
		
\begin{fact} \label{f3}
			Suppose that $l\in X_i$ for $i\in[k]$. Then vectors $\mathbf v_j$, $j\in \{i\}\cup [k+1,n]\setminus \{l\}$, are linearly independent.
		\end{fact}
		
\begin{proof}[Proof of Fact \ref{f3}] 
A simple exercise.
\end{proof}
If $X_i=\emptyset$ for some $i\in [k]$ then $\mathbf v_i=\mathbf 0$, i.e. we get the contradiction because eigenvectors are non-trivial vectors. If $X_i=\{m\}$ for some $i\in [k], m\in [k+1, n]$ then $\mathbf v_i=\mu_{i,m} \mathbf v_m$, i.e. $\mathbf v_i$ is a common eigenvector for all elements of $\mathcal{A}$. Therefore, $|X_i|\geqslant 2$,  thus without loss of generality we have $X_{1,2}=X_1\cap X_2\ne \emptyset$. Indeed, otherwise
$$2k\leqslant \left|\bigcup_{i=1}^k X_i\right|\leqslant |[k+1,n]|= n-k, \text{ i.e. } $$ 
$$n-d\leqslant k\leqslant n/3,\ 2n/3\leqslant d,$$
		the last inequality contradicts our assumption $n\geqslant \lfloor3d/2\rfloor+1$.
		
Fix $l\in X_{1,2}=X_1\cap X_2\ne 0$. Summing (\ref{e3}) for $i=1$ and $i=2$ with proper coefficients, we get the following equality
\begin{multline}\label{e4}
		\mu_{2,l}\mathbf v_1 -\mu_{1,l}\mathbf v_2=\sum_{j\in X_1\setminus X_{1,2}} \mu_{2,l}\mu_{1,j} \mathbf v_j-\sum_{j\in X_2\setminus X_{1,2}} \mu_{1,l}\mu_{2,j} \mathbf v_j+\\
		+\sum_{j\in X_{1,2}}(\mu_{2,l}\mu_{1,j}-\mu_{1,l}\mu_{2,j}) \mathbf v_j.
\end{multline}
		Denote by $X_0$ the set of such $j\in X_{1,2}$ that $\mu_{2,l}\mu_{1,j}=\mu_{1,l}\mu_{2,j}$, i.e. the set of such $j\in X_{1,2}$ that coefficients corresponding $\mathbf v_j$ in (\ref{e4}) are equal to $0$. Note that $l\in X_0.$ 
		Let us write (\ref{e4}) in the following way
		\begin{equation}\label{e5}
		\mathbf v_1=\sum_{j\in (\{2\}\cup X_1\cup X_2)\setminus X_0} \beta_j \mathbf v_j,
		\end{equation}
		where $\beta_j\ne 0$ for all $j\in (\{2\}\cup X_1\cup X_2)\setminus X_0$. Applying $A_l$ to (\ref{e5}) and using that $\mathbf v_j$, $j\in (\{1,2\}\cup X_1\cup X_2)\setminus X_0$, are eigenvectors of $A_l$ we get
		$$\lambda_{1,l} \mathbf v_1=\sum_{j\in (\{2\}\cup X_1\cup X_2)\setminus X_0} \lambda_{j,l}\beta_j \mathbf v_j,$$
		where $\lambda_{j,l}$, $j\in (\{1,2\}\cup X_1\cup X_2)\setminus X_0$, is the eigenvalue corresponding the operator $A_l$ and $\mathbf v_j$.
		By Fact \ref{f3} vectors $\mathbf v_j$, $j\in (\{2\}\cup X_1\cup X_2)\setminus X_0\subset (\{2\}\cup [k+1,n])\setminus \{l\}$ are linearly independent therefore we have that
		$A_l$ is a scalar operator on $span (\{\mathbf v_j: j\in (\{1,2\}\cup X_1\cup X_2)\setminus X_0\})$ associated with the eigenvalue $\lambda=\lambda_{1,l}$, i.e.
\begin{gather*} 
\mathbf w=\mu_{2,l}\mathbf v_1 -\sum_{j\in X_1\setminus X_0} \mu_{2,l}\mu_{1,j}\mathbf v_j=\sum_{j\in X_0} \mu_{2,l}\mu_{1,j} \mathbf v_{j}=\\
=\sum_{j\in X_0} \mu_{1,l}\mu_{2,j} \mathbf v_{j}=
\mu_{1,l}\mathbf v_2 -\sum_{j\in X_2\setminus X_0} \mu_{1,l}\mu_{2,j}\mathbf v_j
\end{gather*}
is also an eigenvector of $A_l$ associated with the eigenvalue $\lambda$. Analogously, $\mathbf w$ is an eigenvector of $A_{l'}$ for every  $l'\in X_0$. Because of $\mathbf w\in H_1\cap H_2$ thus by Fact \ref{f2} $\mathbf w$ is an eigenvector of $A_j$ for any $j\in [n]\setminus X_{1,2}$. 

Lastly, let us show that $\mathbf w$ is an eigenvector of any $A_m$, $m\in X_{1,2}\setminus X_0$. 
Again summing \eqref{e3} for $i=1$ and $i=2$ with proper coefficients, we have
\begin{multline}\label{e6}
\mu_{2,m} \mathbf v_1-\mu_{1,m} \mathbf v_2=\sum_{j\in X_1\setminus X_{1,2}} \mu_{2,m}\mu_{1,j} \mathbf v_j-\sum_{j\in X_2\setminus X_{1,2}} \mu_{1,m}\mu_{2,j} \mathbf v_j+\\
+\sum_{j\in X_{1,2}}(\mu_{2,m}\mu_{1,j}-\mu_{1,m}\mu_{2,j}) \mathbf v_j.
\end{multline}
Denote by $X'_0$ the set of such $j\in X_{1,2}$ that $\mu_{2,m}\mu_{1,j}=\mu_{1,m}\mu_{2,j}$, i.e. the set of such $j\in X_{1,2}$ that coefficients corresponding $\mathbf v_j$ in (\ref{e6}) are equal to $0$. Note that $X_0\cap X_0'=\emptyset$, because otherwise we get $\mu_{2,m}\mu_{1,l'}=\mu_{1,m}\mu_{2,l'}$ and $\mu_{2,l} \mu_{1,l'}=\mu_{1,l}\mu_{2,l'}$, where $l'\in X_0\cap X_0'$, thus $\mu_{2,m}\mu_{1,l}=\mu_{1,m}\mu_{2,l}$, i.e. $m\in X_0$. We now apply the argument  used for \eqref{e4} again, with $l$ replaced by $m$, to obtain that $A_m$ is a scalar operator on $span(\{\mathbf v_j:j\in(\{1,2\}\cup X_1\cup X_2)\setminus X_0'\})$, i.e. $A_m$ is a scalar operator on $span(\{\mathbf v_j:j\in X_0\})$. Thus $\mathbf w\in span(\{\mathbf v_j:j\in X_0\})$ is an eigenvector of $A_m$.

Therefore, $\mathbf w$ is a common eigenvector of all operators $A_j$, $j\in [n]$. Theorem~\ref{theorem:hellyeigenvectors} is proved.
    
\section{Proof of Theorem \ref{theorem:hellyinvariansubspaces}} 
\label{section:prooftheoremhellyinvariantsubspaces}
\begin{proof}[Proof of Theorem \ref{theorem:hellyinvariansubspaces} using Lemma \ref{lemma:subsets}]
Suppose that we proved Theorem \ref{theorem:hellyinvariansubspaces} for any family $\mathcal{A}$ containing $A_0$ of size less then $n$, where $n\geqslant 2d$. Let us prove Theorem \ref{theorem:hellyinvariansubspaces} for a family $\mathcal A=\{A_0,A_1\dots, A_{n-1}\}$. By induction hypothesis, for any $j\in[n-1]$ there exists an invariant non-trivial subspace $H_j$ such that it is an invariant for all operators but $A_j$. 

Denote eigenvectors of $A_0$ by $\mathbf v_1,\dots, \mathbf v_d$. Since these vectors are associated with different eigenvalues, they forms a basis of $\mathbb K^d$. Therefore, for any non-trivial invariant subspace $H$ of $A_0$ there exists $I\subseteq[d]$ such that $H=span(\mathbf v_i: i\in I).$ Indeed, assume the contrary, i.e. $H\ne span(\mathbf v_i, i\in I)$ for any $I\subseteq[d]$, i.e. there are vectors in $H$ such that they could be represented as sums of eigenvectors of $A_0$ that do not belong to $H$. Choose such a vector $\mathbf v$ that it has the minimal number of terms in its representation, i.e. \begin{equation}\label{eq:sum}\mathbf v=\sum_{i=1}^k\alpha_i\mathbf v_{l_i},\end{equation} 
where $\mathbf v_{l_i}\not \in H$, $\alpha_{i}\in\mathbb K, \alpha_i\ne0,$ $i\in[k]$, and $k>1$ is minimal. Applying $A_0$ to \eqref{eq:sum} we get that \begin{equation}\label{eq:sum2}
\mathbf v'=\sum_{i=1}^k \alpha_i\lambda_{l_i} \mathbf v_{l_i}\in H,
\end{equation}
where $\lambda_{l_i}\in \mathbb K$, $i\in [k]$, is the eigenvalue of $A_0$ corresponding to the eigenvector $\mathbf v_{l_i}$. Therefore, we have that the vector 
$$\mathbf v'-\lambda_{l_k} \mathbf v=\sum_{i=1}^{k-1} \alpha_i(\lambda_{l_i}-\lambda_{l_k})\mathbf v_{l_i}\in H$$ 
can be represented as the sum of $k-1$ eigenvectors of $A_0$ that do not belong to $H$. This contradiction shows that for each $i$, $i\in[n-1]$, we can assign a non-empty subset $M_i\subset [d]$ such that $H_i=span(\mathbf v_i:i\in M_i)$ because $H_i$ are invariant subspaces of $A_0$. By Lemma~\ref{lemma:subsets} there exists a non-empty set $I\subset [d]$ such that  $$M=\bigcup_{i\in I}M_i= \bigcup_{i\in I\setminus\{j\}} M_i\ne [d]$$ for any $j\in I$, i.e. $span(\mathbf v_j:j\in M)\ne \mathbb K^d$ is a common non-trivial invariant subspace of all operators in the family~$\mathcal A$.
\end{proof}
\begin{proof}[Proof of Lemma \ref{lemma:subsets}]
The proof is by induction on $q$. For $q=1$, there is nothing to prove. Assume that the inequality $p\leq 2q-2$ is proved for $q< n$, let us prove it for $q=n$. A non-empty subset $I\subseteq[p]$ is called \emph{interesting} if $\bigcup_{i\in I} M_i\ne M_j$ for any $j\in [p]$ and $\bigcup_{i\in I}M_i\ne [q]$. Assume that there is at least one interesting subset and choose an interesting subset $J$ that has the maximal cardinality. Note that there is $k\in J$ such that $M_k$ contains an \emph{unique element}, i.e. there is $x\in M_k$ such that $x\not \in M_j$ for any $j\in I\setminus \{k\}$. Let us consider the following family $\mathcal{N} =\{N_1, \dots, N_p\}$, where $N_k=\bigcup_{i\in J} M_i$ and $N_i=M_i$ for $i\ne k$. Note that for any $l \in [p]\setminus J$ we have $N_k\subset N_l$ or $N_k\cup N_l=[q]$. In the converse case, we get the interesting subset $J\cup \{l\}$ that should have the larger cardinality then $J$. Let us check that the condition of the lemma holds for the new family. Trivially, it holds for any $I\subseteq [q]\setminus\{k\}$ and $I\subseteq J$ (because $x\in N_k$ and $x\not \in N_i$, $i\in J$). Let us check that the condition holds for $I$ such that $k, l\in I$, where $l\not \in J$. We know that the condition holds for the family $\mathcal{M}$ and $I'=I\cup J$. If $\bigcup_{i\in I'} M_i=[q]$ then $\bigcup_{i\in I}N_i=[q]$. If there is $k'\in I'$ such that $M_{k'}$ contains an unique element $y$ then, since $N_k=\bigcup_{i\in I}M_i\subset N_l=M_l$, we get $k'\not\in J$. It is easy to check that the condition holds for the family $\mathcal N$ and $I$ because $N_{k'}=M_{k'}$ has the unique element $y$.

We make such replacements of the family till it is possible. Obviously, we can not do them infinitely many times because after each replacement the total cardinality of sets increases. Finally, we get a family $\mathcal{M'}$ that for every two subsets in $\mathcal{M}'$ one contains another or their union is $[q]$. Without loss of generality assume that a subset $A=\{1,\dots, t\}\in \mathcal{M}'$ such that it does not contain another subset in $\mathcal{M}$. Denote by $K_1, \dots, K_l\in \mathcal{M}$ subsets that do not contain $A$, i.e. $[q]\setminus A\subseteq K_i$. We can certainly assume that $K_l\cap A=\emptyset$ (if for every $K_i$ we have $K_i\cap A\ne \emptyset$ then we can add the subset $[q]\setminus A$ to the family $\mathcal M'$). Note that the family $\{K_i\cap A: i\in[l-1]\}$ (where the role of $q$ is played by $k$) satisfies condition of the lemma. By the induction hypothesis, we get $l-1\leqslant 2k-2$, i.e. $l\leqslant 2k-1$. If we delete $A$ and $K_i$ from the family $\mathcal M'$ (i.e. at most $2k$ subsets) and delete elements $1,\dots , k$ from other subsets in $\mathcal{M}'$ then the new family satisfies the condition of the lemma, i.e. it has at most $2(q-k)-2$ subsets. Thus the original family $\mathcal{M}$ contains at most $2k+2(q-k)-2=2q-2$ subsets.
\end{proof}
\emph{Example.} Obviously, the following example shows that the inequality in Lemma~\ref{lemma:subsets} is sharp:
\begin{gather*}
\mathcal{M}=\{\{1\},\{1,2\}, \{1,2,3\},\dots,\{1,\dots, q-1\}\}\cup \\
\cup\{\{q\},\{q,q-1\},\{q,q-1,q-2\}, \dots,\{q,\dots,2\}\}.
\end{gather*}

\end{document}